\documentclass[reqno]{amsart}

\usepackage{amsmath}
\usepackage{amsthm}
\usepackage{amsfonts}
\usepackage{amssymb}
\usepackage{enumerate}
\usepackage{graphicx}

\newcommand{\indicator}[1]{\ensuremath{\mathbf{1}_{\{#1\}}}}
\newcommand{\oindicator}[1]{\ensuremath{\mathbf{1}_{{#1}}}}

\DeclareMathOperator{\tr}{tr}

\DeclareMathOperator{\diag}{diag}

\newcommand{\Prob}{\mathbb{P}}
\newcommand{\E}{\mathbb{E}}

\renewcommand\Re{\operatorname{Re}}
\renewcommand\Im{\operatorname{Im}}
\newcommand{\eps}{\varepsilon}

\newcommand{\U}{\mathrm{U}}
\newcommand{\Or}{\mathrm{O}}
\newcommand{\SO}{\mathrm{SO}}
\newcommand{\Sp}{\mathrm{Sp}}

\theoremstyle{plain}
  \newtheorem{theorem}{Theorem}

  \newtheorem{lemma}[theorem]{Lemma}
  \newtheorem{corollary}[theorem]{Corollary}

\theoremstyle{definition}
  \newtheorem{definition}[theorem]{Definition}
  
  \newtheorem{remark}[theorem]{Remark}

\begin{document}
\title[Critical points of random polynomials]{Critical points of random polynomials and characteristic polynomials of random matrices}

\author[S. O'Rourke]{Sean O'Rourke}
\address{Department of Mathematics, University of Colorado at Boulder, Boulder, CO 80309  }
\email{sean.d.orourke@colorado.edu}

\begin{abstract}
Let $p_n$ be the characteristic polynomial of an $n \times n$ random matrix drawn from one of the compact classical matrix groups.  We show that the critical points of $p_n$ converge to the uniform distribution on the unit circle as $n$ tends to infinity.  More generally, we show the same limit for a class of random polynomials whose roots lie on the unit circle.  Our results extend the work of Pemantle--Rivin \cite{PR} and Kabluchko \cite{K} to the setting where the roots are neither independent nor identically distributed.   
\end{abstract}

\maketitle

\section{Introduction}

A \emph{critical point} of a polynomial $f$ is a root of its derivative $f'$.  There are many results concerning the location of critical points of polynomials whose roots are known.  For example, the famous Gauss--Lucas theorem offers a geometric connection between the roots of a polynomial and the roots of its derivative.  

\begin{theorem}[Gauss--Lucas; Theorem 6.1 from \cite{M}] \label{thm:gauss}
If $f$ is a non-constant polynomial with complex coefficients, then all zeros of $f'$ belong to the convex hull of the set of zeros of $f$.  
\end{theorem}

There are also a number of refinements of Theorem \ref{thm:gauss}.  We refer the reader to \cite{Azero,Bzero,CM,Dzero,Drozero,GRR,Jzero,Mzero,Malzero,Mcon,Pzero,Rzero,Szero,Szero2,Smahler,Tzero} and references therein.  

Pemantle and Rivin \cite{PR} initiated the study of a probabilistic version of the Gauss--Lucas theorem.  In order to introduce their results, we fix the following notation.  For a polynomial $f$ of degree $n$, we define the empirical measure constructed from the roots of $f$ as 
$$ \mu_{f} := \frac{1}{n} \sum_{z \in \mathbb{C} : f(z) = 0} \mathcal{N}_f(z) \delta_{z}, $$
where $\mathcal{N}_f(z)$ is the multiplicity of the zero at $z$ and $\delta_z$ is the unit point mass at $z$.  

For an integer $k \geq 1$, we use the convention that
$$ \mu_{f}^{(k)} := \mu_{f^{(k)}}. $$
That is, $\mu_{f}^{(k)}$ is the empirical measure constructed from the roots of the $k$-th derivative of $f$.  Similarly, we write $\mu_f'$ to denote the empirical measure constructed from the critical points of $f$.   

Let $X_1, X_2,\ldots$ be independent and identically distributed (iid) random variables taking values in $\mathbb{C}$.  Let $\mu$ be the probability distribution of $X_1$.  For each $n \geq 1$, consider the polynomial
\begin{equation} \label{eq:pnprod}
	p_n(z) := (z-X_1) \cdots (z-X_n). 
\end{equation}
Pemantle and Rivin \cite{PR} show, assuming $\mu$ has finite one-dimensional energy, that $\mu'_{p_n}$ converges weakly to $\mu$ as $n$ tends to infinity.  

Let us recall what it means for a sequence of random probability measures to converge weakly. 

\begin{definition}[Weak convergence of random probability measures]
Let $T$ be a topological space (such as $\mathbb{R}$ or $\mathbb{C}$), and let $\mathcal{B}$ be its Borel $\sigma$-field.  Let $(\mu_n)_{n\geq 1}$ be a sequence of random probability measures on $(T,\mathcal{B})$, and let $\mu$ be a probability measure on $(T,\mathcal{B})$.  We say \emph{$\mu_n$ converges weakly to $\mu$ in probability} as $n \to \infty$ (and write $\mu_n \to \mu$ in probability) if for all bounded continuous $f:T \to \mathbb{R}$ and any $\eps > 0$,
$$ \lim_{n \to \infty} \Prob \left( \left| \int f d\mu_n - \int f d\mu \right| > \eps \right) = 0. $$
In other words, $\mu_n \to \mu$ in probability as $n \to \infty$ if and only if $\int f d\mu_n \to \int f d\mu$ in probability for all bounded continuous $f: T \to \mathbb{R}$.  Similarly, we say \emph{$\mu_n$ converges weakly to $\mu$ almost surely} as $n \to \infty$ (and write $\mu_n \to \mu$ almost surely) if for all bounded continuous $f:T \to \mathbb{R}$,
$$ \lim_{n \to \infty} \int f d\mu_n = \int f d\mu $$
almost surely.   
\end{definition}

Kabluchko \cite{K} generalized the results of Pemantle and Rivin to the following.  

\begin{theorem}[Kabluchko] \label{thm:PMK}
Let $\mu$ be any probability measure on $\mathbb{C}$.  Let $X_1, X_2, \ldots$ be a sequence of iid random variables with distribution $\mu$.  For each $n \geq 1$, let $p_n$ be the degree $n$ polynomial given in \eqref{eq:pnprod}.  Then $\mu'_{p_n}$ converges weakly to $\mu$ in probability as $n \to \infty$.  
\end{theorem}

The following corollary of Theorem \ref{thm:PMK} will be relevant to this note.

\begin{corollary} \label{cor:PMK}
Let $\theta_1, \theta_2, \ldots$ be a sequence of iid random variables distributed uniformly on $[0,2\pi)$.  For each $n \geq 1$, let 
$$ p_n(z) := \prod_{j=1}^n (z - e^{i \theta_j}). $$
Then $\mu'_{p_n}$ converges in probability to the uniform probability distribution on the unit circle centered at the origin in the complex plane as $n \to \infty$.  
\end{corollary}

Corollary \ref{cor:PMK} also follows from the work of Subramanian in \cite{S}.

\section{Main Results}

The goal of this note is to prove a version of Theorem \ref{thm:PMK} when the random variables $X_1, X_2, \ldots$ are neither independent nor identically distributed.  Of particular interest will be the case when the roots of $p_n$ are eigenvalues of a random matrix.  

The \emph{eigenvalues} of a square matrix $M$ are the zeros of its characteristic polynomial $p_M(z) := \det(zI - M)$, where $I$ denotes the identity matrix.  We let $\mu_M$ denote the empirical spectral measure of $M$.  That is, $\mu_M$ is the empirical measure constructed from the roots of the characteristic polynomial $p_M$.  Similarly, we let $\mu_M'$ be the empirical measure constructed from the roots of $p'_M$.  

As a motivating example, we begin with the case when $M$ is Hermitian.

\subsection{Characteristic polynomials of Hermitian random matrices}

If the matrix $M$ is Hermitian (that is, $M = M^\ast$, where $M^\ast$ denotes the conjugate transpose of $M$), then the eigenvalues of $M$ are real and $\mu_M$ is a probability measure on the real line.  In this case, Theorem \ref{thm:hermitian} below describes the well-known connection between $\mu_M$ and $\mu_M'$.  Before stating the result, we first recall the following definition.  

\begin{definition}[L\'{e}vy distance]
Let $\mu$ and $\nu$ be two probability measures on the real line with cumulative distribution functions $F$ and $G$ respectively.  Then the \emph{L\'{e}vy distance} $L(\mu, \nu)$ between $\mu$ and $\nu$ is given by
$$ L(\mu, \nu) := \inf \{ \eps \geq 0 : G(x - \eps) - \eps \leq F(x) \leq G(x + \eps) + \eps \text{ for all } x \in \mathbb{R} \}. $$
\end{definition}

It is well-known, for measures on the real line, that convergence in L\'{e}vy distance is equivalent to convergence in distribution; we refer the reader to \cite[Chapter 13.2]{Kprob} and \cite[Exercise 13.2.6]{Kprob} for further details.  

\begin{theorem} \label{thm:hermitian}
For each $n \geq 1$, let $X_n$ be a $n \times n$ random Hermitian matrix.  Then $\mu_{X_n}'$ is a random probability measure on the real line and 
$$ L(\mu_{X_n}, \mu_{X_n}') \longrightarrow 0 $$
almost surely as $n \to \infty$
\end{theorem}
\begin{proof}
Since the eigenvalues of $X_n$ are real, the Gauss--Lucas theorem (Theorem \ref{thm:gauss}) guarantees that $\mu'_{X_n}$ is a probability measure on the real line.  

Let $I \subset \mathbb{R}$ be an interval.  Let $N_I$ denote the number of zeros of $p_{X_n}$ in $I$ (i.e. the number of eigenvalues of $X_n$ in $I$), and let $N_I'$ denote the number of critical points of $p_{X_n}$ in $I$.  Since the zeros of $p'_{X_n}$ interlace the zeros of $p_{X_n}$, we have
$$ \left| N_I - N_I' \right| \leq 1, $$
and the claim follows from \cite[Lemma B.18]{BSbook}.  
\end{proof}

As a concrete example, we present the following corollary for Wigner random matrices.  

\begin{corollary} \label{cor:wigner}
Let $\xi$ be a complex-valued random variable with unit variance, and let $\zeta$ be a real-valued random variable.  For each $n \geq 1$, let $X_n$ be a $n \times n$ Hermitian matrix whose diagonal entries are iid copies of $\zeta$, those above the diagonal are iid copies of $\xi$, and all the entries on and above the diagonal are independent.  Then $\mu_{\frac{1}{\sqrt{n}} X_n}'$ is a probability measure on the real line and 
$$ \mu_{\frac{1}{\sqrt{n}} X_n}' \longrightarrow \mu_{\mathrm{sc}} $$ 
almost surely as $n \to \infty$, where $\mu_{\mathrm{sc}}$ is the measure on the real line with density
\begin{equation*} 
	\rho_{\mathrm{sc}}(x) := \left\{
     		\begin{array}{ll}
		\frac{1}{2 \pi} \sqrt{4-x^2}, &|x| \leq 2\\
		0, &|x| > 2.
		\end{array}
   	\right. 
\end{equation*}
\end{corollary}
\begin{proof}
In view of \cite[Theorem 2.5]{BSbook} and \cite[Exercise 13.2.6]{Kprob}, we have
$$ L \left( \mu_{\frac{1}{\sqrt{n}} X_n}, \mu_{\mathrm{sc}} \right) \longrightarrow 0 $$
almost surely as $n \to \infty$.  Thus, the claim follows from Theorem \ref{thm:hermitian} by applying the triangle inequality for L\'{e}vy distance.   
\end{proof}

The same arguments can also be used to generalize Theorem \ref{thm:hermitian} and Corollary \ref{cor:wigner} to higher-order derivatives.

\subsection{Random matrices from the compact classical groups}

In this note, we extend Theorem \ref{thm:hermitian} to random matrices which are not Hermitian.  In particular, we consider random matrices distributed according to Haar measure on the compact classical matrix groups.  We begin by recalling some definitions.  

\begin{definition}[Compact classical matrix groups] \hfill
\begin{enumerate}
\item An $n \times n$ matrix $M$ over $\mathbb{R}$ is \emph{orthogonal} if 
$$ M M^\mathrm{T} = M^\mathrm{T} M = I_n, $$
where $I_n$ denotes the $n \times n$ identity matrix and $M^\mathrm{T}$ is the transpose of $M$.  The set of $n \times n$ orthogonal matrices over $\mathbb{R}$ is denoted by $\Or(n)$.  
\item The set $\SO(n) \subset \Or(n)$ of \emph{special orthogonal matrices} is defined by 
$$ \SO(n) := \{ M \in \Or(n) : \det(M) = 1 \}. $$
\item An $n \times n$ matrix $M$ over $\mathbb{C}$ is \emph{unitary} if
$$ M M^\ast = M^\ast M = I_n, $$
where $M^\ast$ denotes the conjugate transpose of $M$.  The set of $n \times n$ unitary matrices over $\mathbb{C}$ is denoted $\U(n)$.  
\item If $n$ is even, we say an $n \times n$ matrix $M$ over $\mathbb{C}$ is \emph{symplectic} if $M \in \U(n)$ and
$$ M J M^\ast = M^\ast J M = J, $$
where
$$ J := \begin{bmatrix} 0 & I_{n/2} \\ -I_{n/2} & 0 \end{bmatrix}. $$
The set of $n \times n$ symplectic matrices over $\mathbb{C}$ is denoted $\Sp(n)$.  
\end{enumerate}
\end{definition}

Recall that if $M$ is a matrix from one of the compact matrix groups introduced above, then the eigenvalues of $M$ all lie on the unit circle in the complex plane centered at the origin.  

For any compact Lie group $G$, there exists a unique translation-invariant probability measure on $G$ called \emph{Haar measure}; see, for example, \cite[Chapter 2.2]{F}.  In this note, we will be interested in the case when $G$ is one of classical compact matrix groups defined above.  

For the compact matrix groups, there are a number of intuitive ways to describe a matrix distributed according to Haar measure.  Recall that a complex standard normal random variable $Z$ can be represented as $Z = X + iY$, where $X$ and $Y$ are independent real normal random variables with mean zero and variance $1/2$.   Form an $n \times n$ random matrix with independent complex standard normal entries and perform the Gram--Schmidt algorithm on the columns.  The result is a random unitary matrix distributed according to Haar measure on $\U(n)$.  Indeed, invariance follows from the invariance of complex Gaussian random vectors under $\U(n)$.  Similar Gaussian constructions yield random matrices distributed according to Haar measure on the other compact matrix groups.  

We now present our main result for the classical compact matrix groups.  

\begin{theorem} \label{thm:compact}
For each $n \geq 1$, let $M_n$ be an $n \times n$ matrix Haar distributed on $\Or(n)$, $\SO(n)$, $\U(n)$, or $\Sp(n)$.  Then $\mu_{M_n}'$ converges in probability as $n \to \infty$ to the uniform probability distribution on the unit circle centered at the origin.  
\end{theorem}

\begin{remark}
If $M_n$ is an $n \times n$ random matrix Haar distributed on $\Or(n)$, $\SO(n)$, $\U(n)$, or $\Sp(n)$, then $\mu_{M_n}$ also converges in probability to the uniform distribution on the unit circle centered at the origin.  Moreover, in \cite{MM}, the authors prove that the convergence holds in the almost sure sense and give a rate of convergence.  
\end{remark}

Figure \ref{fig:orthogonal} depicts a numerical simulation of the zeros and critical points of the characteristic polynomial of a random orthogonal matrix chosen according to Haar measure.  

\begin{figure}[ht]
	\begin{center}
	\includegraphics[trim=3cm 7cm 3cm 7cm,clip=true,width=10cm, height=10cm]{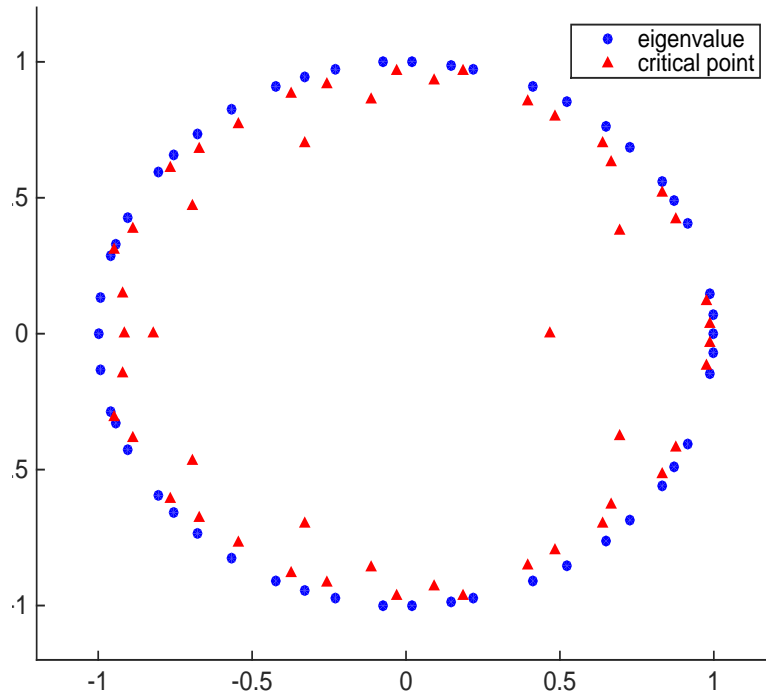}
   	\caption{The zeros and critical points of the characteristic polynomial of a random orthogonal matrix of size $50 \times 50$.}
	\label{fig:orthogonal}
	\end{center}
\end{figure}

\subsection{Random polynomials with roots on the unit circle}
More generally, we consider random polynomials of the form 
$$ p_n(z) := \prod_{j=1}^n (z - X_j), $$
where $X_1, X_2, \ldots$ are random variables on the unit circle, not necessarily independent or identically distributed.  Indeed, we will deduce Theorem \ref{thm:compact} from the following more general result.

\begin{theorem} \label{thm:main}
For each $n \geq 1$, let $\theta_1^{(n)},\ldots, \theta_n^{(n)}$ be random variables on $[0,2\pi)$.  Set 
$$ p_n(z) := \prod_{j=1}^n (z - e^{i \theta_j^{(n)}}). $$
Assume
\begin{enumerate}[(i)]
\item \label{cond:main:1} we have
$$ \lim_{\delta \searrow 0} \limsup_{n \to \infty} \Prob \left( \left| \sum_{j=1}^n e^{-i \theta_j^{(n)}} \right| \leq \delta \right) = 0, $$
\item \label{cond:main:2} for almost every $z \in \mathbb{D} := \{ w \in \mathbb{C} : |w| < 1 \}$,
$$ \lim_{\delta \searrow 0} \limsup_{n \to \infty} \Prob \left( \left| \sum_{m=0}^{\lfloor \log^2 n \rfloor} z^m \sum_{j=1}^n e^{-i \theta_j^{(n)} (m+1)} \right| \leq \delta \right) = 0, $$
\item \label{cond:main:3} for all integers $m \geq 1$,
$$ \frac{1}{n} \sum_{j=1}^n e^{i \theta_j^{(n)}m} \longrightarrow 0 $$
in probability as $n \to \infty$.  
\end{enumerate}
Then $\mu_{p_n}'$ converges in probability as $n \to \infty$ to the uniform probability distribution on the unit circle centered at the origin.  
\end{theorem}

We pause for a moment to discuss the three assumptions of Theorem \ref{thm:main}.  Roughly speaking, condition \eqref{cond:main:3} is the most important, while conditions \eqref{cond:main:1} and \eqref{cond:main:2} are technical anti-concentration estimates.  Indeed, condition \eqref{cond:main:3} implies that the empirical measure constructed from $e^{i\theta_1^{(n)}}, \ldots, e^{i \theta_n^{(n)}}$ converges in probability as $n \to \infty$ to the uniform distribution on the unit circle centered at the origin.  We also note that the sequence $\log^2 n$ appearing in condition \eqref{cond:main:2} is not vital; it can be replaced with $(\log n)^{1 + \eps}$ for any $\eps > 0$.  

We will also verify the following alternative formulation of Theorem \ref{thm:main}.

\begin{theorem}[Alternative formulation] \label{thm:alternative}
For each $n \geq 1$, let $\theta_1^{(n)},\ldots, \theta_n^{(n)}$ be random variables on $[0,2\pi)$.  Set 
$$ p_n(z) := \prod_{j=1}^n (z - e^{i \theta_j^{(n)}}). $$
Assume
\begin{enumerate}[(i)]
\item \label{cond:alt:1} we have
\begin{equation*} 
	\lim_{\delta \searrow 0} \limsup_{n \to \infty} \Prob \left( \left| \sum_{j=1}^n e^{-i \theta_j^{(n)}} \right| \leq \delta \right) = 0, 
\end{equation*}
\item \label{cond:alt:2} for almost every $z \in \mathbb{D} := \{ w \in \mathbb{C} : |w| < 1 \}$,
\begin{equation*} 
	\frac{1}{n} \log \left| \sum_{j=1}^n \frac{1}{z - e^{i \theta_j^{(n)}}} \right| \longrightarrow 0 
\end{equation*}
in probability as $n \to \infty$,
\item \label{cond:alt:3} for all integers $m \geq 1$,
\begin{equation*} 
	\frac{1}{n} \sum_{j=1}^n e^{i \theta_j^{(n)}m} \longrightarrow 0 
\end{equation*}
in probability as $n \to \infty$.  
\end{enumerate}
Then $\mu_{p_n}'$ converges in probability as $n \to \infty$ to the uniform probability distribution on the unit circle centered at the origin.  
\end{theorem}

We will use Theorem \ref{thm:main} to prove Theorem \ref{thm:compact}.  However, Theorem \ref{thm:alternative} is also useful.  For example, we can recover Corollary \ref{cor:PMK} from Theorem \ref{thm:alternative}.  Indeed, if $\theta_1^{(n)}, \ldots, \theta_n^{(n)}$ are iid random variables uniformly distributed on $[0,2 \pi)$, then the assumptions of Theorem \ref{thm:alternative} can be verified using \cite[Theorem 2.22]{P}, \cite[Lemma 2.1]{K}, and the law of large numbers.  Theorem \ref{thm:alternative} is also useful when the random variables $\theta_1^{(n)}, \ldots, \theta_n^{(n)}$ are dependent.  To illustrate this point, we will use Theorem \ref{thm:alternative} to verify the following corollary.

\begin{corollary} \label{cor:dependent}
Let $\theta_1, \theta_2, \ldots$ be a sequence of iid random variables distributed uniformly on $[0,2 \pi)$.  For each $n \geq 1$, set
$$ p_{2n}(z) := \prod_{j = 1}^{n} (z - e^{i \theta_j})(z - e^{- i \theta_j}). $$
Then $\mu_{p_{2n}}'$ converges in probability as $n \to \infty$ to the uniform probability distribution on the unit circle centered at the origin. 
\end{corollary}

\subsection{Discussion and open problems}
We conjecture that for many classes of random polynomials the critical points should be stochastically close to the distribution of the roots.  Intuitively, this would imply that the distribution of the critical points is nearly identical to the distribution of the roots for a ``typical'' polynomial of high degree.

As another example, consider the Kac polynomials.  In this case, one can show even more by applying the results of Kabluchko and Zaporozhets \cite{KZ}.  

\begin{theorem}[Kabluchko--Zaporozhets]
Let $\xi_0, \xi_1, \ldots$ be a sequence of non-degenerate iid random variables such that $\E \log(1 + |\xi_0|) < \infty$.  For each $n \geq 1$, let $f_n(z) = \sum_{j=0}^n \xi_j z^j$.  Fix an integer $k \geq 1$.  Then $\mu_{f_n}$ and $\mu_{f_n}^{(k)}$ both converge in probability as $n \to \infty$ to the uniform probability distribution on the unit circle centered at the origin.  
\end{theorem}
\begin{proof}
Both claims follow from \cite[Theorem 2.2]{KZ} by simply estimating the coefficients of $f_n^{(k)}$.  In fact, a similar argument allows one to consider solutions of the equation $f_n^{(k)} = c_n$, where $c_n$ is a constant; see \cite[Remark 2.11]{KZ} for details.  
\end{proof}

We conjecture that this universality phenomenon should also hold for the characteristic polynomial of many random matrix ensembles.  For instance, Figure \ref{fig:ginibre} depicts a numerical simulation of the zeros and critical points of the characteristic polynomial of a random matrix with iid real standard normal entries.  

\begin{figure}[ht]
	\begin{center}
	\includegraphics[trim=3cm 7cm 2cm 7cm,clip=true,width=6cm,height=6cm]{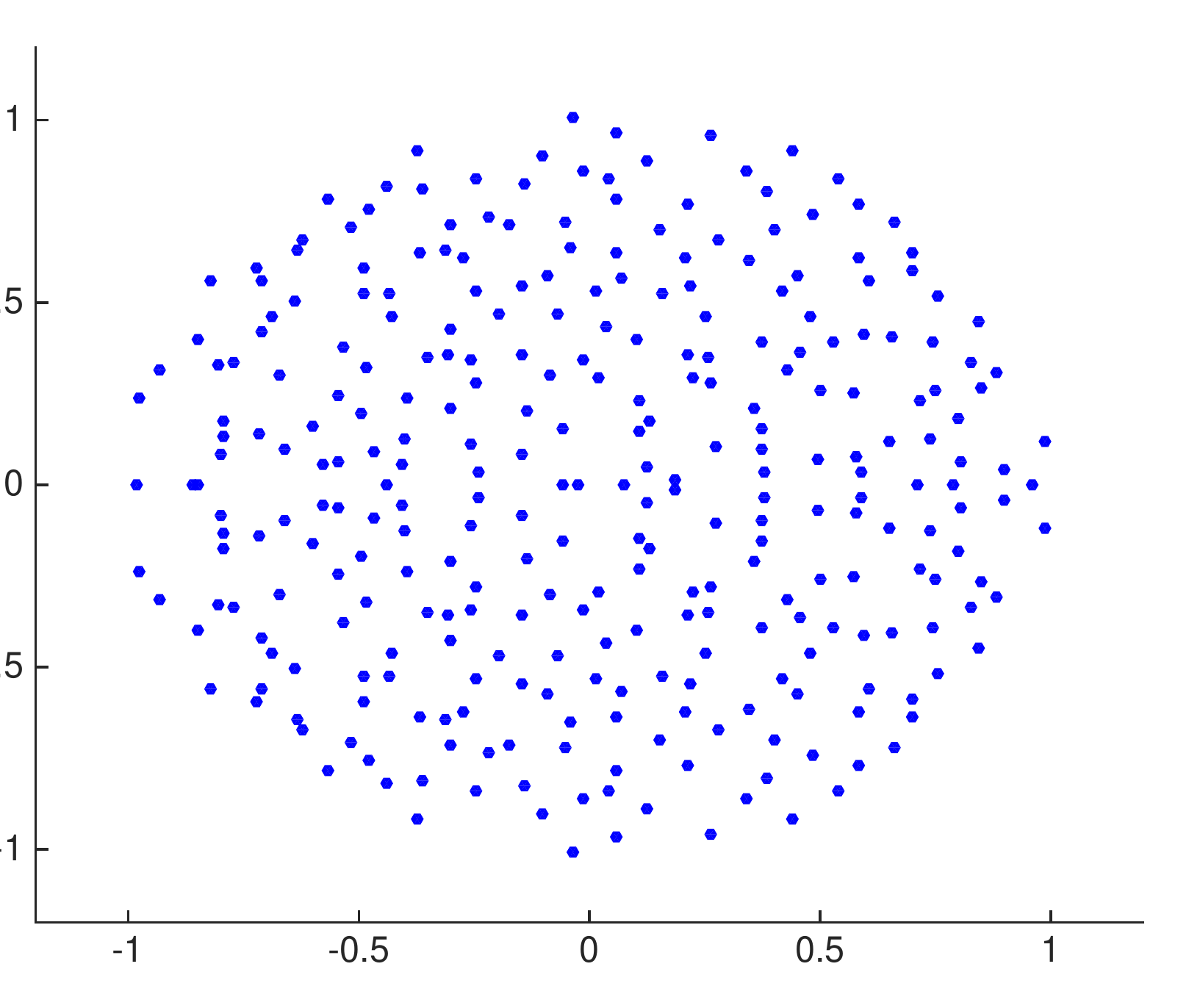}
	\includegraphics[trim=2cm 7cm 3cm 7cm,clip=true,width=6cm,height=6cm]{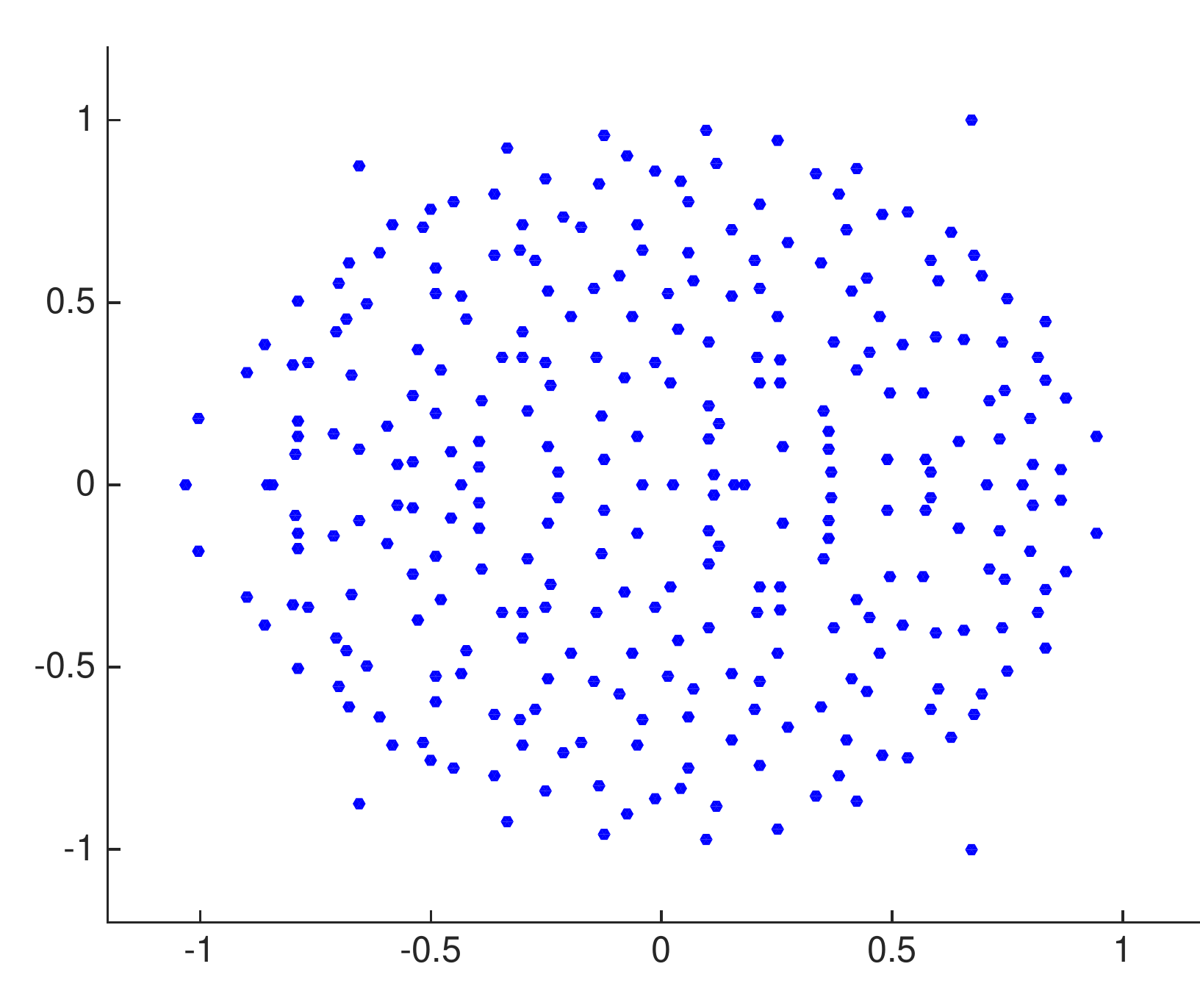}
   	\caption{The roots and critical points of the characteristic polynomial of an $n \times n$ random matrix with iid real standard normal entries when $n=300$.  The figure on the left depicts the location of the eigenvalues (scaled by $1/\sqrt{n}$).  The figure on the right shows the location of the critical points (scaled by $1/\sqrt{n}$).}
	\label{fig:ginibre}
	\end{center}
\end{figure}

\subsection{Organization}

The paper is organized as follows.  In Section \ref{sec:compact}, we prove Theorem \ref{thm:compact} and Corollary \ref{cor:dependent} using Theorems \ref{thm:main} and \ref{thm:alternative}.  The proof of Theorems \ref{thm:main} and \ref{thm:alternative} is contained in Sections \ref{sec:proof} and \ref{sec:lemmas}.

\subsection{Notation}

We let $\mathbb{D}_r := \{ z \in \mathbb{C} : |z| < r\}$ be the open disk of radius $r>0$ centered at the origin and $\overline{\mathbb{D}}_r := \{z \in \mathbb{C} : |z| \leq r \}$ its closure.  We write $\mathbb{D} := \mathbb{D}_1$.  

We let $C$ and $K$ denote constants that are non-random and may take on different values from one appearance to the next.  The notation $K_p$ means that the constant $K$ depends on another parameter $p$.  

We write a.s., a.a., and a.e. for almost surely, Lebesgue almost all, and Lebesgue almost everywhere respectively.  For an event $E$, we let $\oindicator{E}$ denote the indicator function of $E$; $E^C$ is the complement of $E$.

\section{Proof of Theorem \ref{thm:compact} and Corollary \ref{cor:dependent}} \label{sec:compact}

In this section, we prove Theorem \ref{thm:compact} and Corollary \ref{cor:dependent} using Theorems \ref{thm:main} and \ref{thm:alternative}.  

\subsection{Proof of Theorem \ref{thm:compact}}

We will apply Theorem \ref{thm:main} to prove Theorem \ref{thm:compact}.  For each $n \geq 1$, let $M_n$ be an $n \times n$ matrix Haar distributed on $\Or(n)$, $\SO(n)$, $\U(n)$, or $\Sp(n)$.  Let $e^{i \theta_1^{(n)}}, \ldots, e^{i \theta_n^{(n)}}$ be the eigenvalues of $M_n$, where $\theta_1^{(n)}, \ldots, \theta_n^{(n)} \in [0, 2\pi)$.  It now suffices to show that the eigenvalues of $M_n$ satisfy the three assumptions of Theorem \ref{thm:main}.  

In order to verify the assumptions of Theorem \ref{thm:main}, we will need the following multivariate central limit theorem for traces of random matrices from the classical matrix groups found in \cite{DSclt,Sclt}.  First, we recall the Wasserstein distance between two probability distributions.  

\begin{definition}[Wasserstein distance]
Let $(S,d)$ be a separable metric space, and let $\mu$ and $\nu$ be two probability measures on $S$.  By $M(\mu,\nu)$ we denote the set of all probability measures on $S \times S$ with marginals $\mu$ and $\nu$.  The \emph{Wasserstein distance} $d_{\mathcal{W}}(\mu,\nu)$ between $\mu$ and $\nu$ is defined by
$$ d_{\mathcal{W}}(\mu,\nu) := \inf \left\{ \int d(x,y) d \pi(x,y) : \pi \in M(\mu,\nu) \right\}. $$
We write $d_{\mathcal{W}}(P,Q)$, where $P$ and $Q$ are two random variables taking values in $S$, to mean the Wasserstein distance between the distributions of $P$ and $Q$.  
\end{definition}

The Kantorovich--Rubinstein theorem gives an equivalent formulation of the Wasserstein distance in terms of Lipschitz functions on the separable metric space $(S,d)$.  We refer the reader to \cite[Section 11.8]{D} for further details.  We now state the results from \cite{DSclt,Sclt}; the case where $M_n$ is drawn according to Haar measure from $\U(n)$, $\SO(n)$, or $\Sp(n)$ is handled in \cite[Theorem 1.1]{DSclt}, while the orthogonal group $\Or(n)$ is studied in \cite[Theorem 5.1]{Sclt}.  

\begin{theorem}[D\"{o}bler--Stolz] \label{thm:CLT}
Let $M_n$ be distributed according to Haar measure on $\Or(n)$, $\SO(n)$, $\U(n)$, or $\Sp(n)$.  For integers $d \geq 1$ and $r = 1, \ldots, d$, consider the $r$-dimensional (complex or real) random vector 
$$ W_{d,r,n} := (f_{d-r+1}(M_n), f_{d - r + 2}(M_n), \ldots, f_d(M_n))^{\mathrm{T}}, $$
where $f_j(M_n) := \tr (M_n^j)$ in the unitary case, 
$$ f_j(M_n) := \left\{
	\begin{array}{lr}
		\tr (M_n^j),& \text{ if } j \text{ is odd}, \\
		\tr (M_n^j) - 1,& \text{ if } j \text{ is even}
	\end{array} \right. $$
in the orthogonal and special orthogonal cases, and
$$ f_j(M_n) := \left\{
	\begin{array}{lr}
		\tr (M_n^j),& \text{ if } j \text{ is odd}, \\
		\tr (M_n^j) + 1,& \text{ if } j \text{ is even}
	\end{array} \right. $$ 
in the symplectic case.  In the orthogonal, special orthogonal, and symplectic cases, let $Z_{r,d} := (Z_{d-r+1}, \ldots, Z_d)^{\mathrm{T}}$ denote an $r$-dimensional real standard normal random vector.  In the unitary case, $Z$ is defined as a standard complex normal random vector.  In all cases take $\Sigma$ to be the diagonal matrix $\diag(d-r+1,d-r+2, \ldots,d)$, and write $Z_{\Sigma,r,d} := \Sigma^{1/2} Z_{r,d}$.  Then there exists an absolute constant $C > 0$ (independent of $r$, $d$, and $n$) such that, for any $n \geq 4d+1$, we have
$$ d_{\mathcal{W}}(W_{n,r,d}, Z_{\Sigma,r,d}) \leq C \frac{ \max \left\{ \frac{r^{7/2}}{(d-r+1)^{3/2}}, (d-r)^{3/2} \sqrt{r} \right\}}{n}. $$
\end{theorem}

\begin{remark}
There is a large collection of literature concerning traces of random elements from the classical compact matrix groups.  We refer the reader to \cite{DE,DS,DSclt,Ful,HR,J,PV,Sosh,Ste,Sclt} and references therein.  
\end{remark}

We now verify the three assumptions of Theorem \ref{thm:main}.  We will use the same notation as in Theorem \ref{thm:CLT}.  Set $N := \lfloor \log^2 n \rfloor$.  By Theorem \ref{thm:CLT}, there exists random variables $\xi_1^{(n)}, \ldots, \xi_{N+1}^{(n)}$ such that, for $n$ sufficiently large, the random vector
$$ (\xi_1^{(n)}, \ldots, \xi_{N+1}^{(n)})^\mathrm{T} $$
has the same distribution as 
$$ (f_1(M_n), \ldots, f_{N+1}(M_n))^\mathrm{T}, $$
and
\begin{equation} \label{eq:CLT:exp}
	\E \sqrt{ \sum_{m=0}^N \left| \xi_{m+1}^{(n)} - \sqrt{m+1} Z_{m+1} \right|^2 } \leq C \frac{(N+1)^{7/2}}{n}, 
\end{equation}
where $Z := (Z_1, \ldots, Z_{N+1})^{\mathrm{T}}$ is a standard normal random vector.  Here $f_j(M_n)$ is defined as in Theorem \ref{thm:CLT}, $C > 0$ is an absolute constant, and $Z$ is a complex standard normal random vector in the unitary case and a real standard normal random vector in the other cases.  

For any positive integer $m$, we write
$$ \tr M_n^{m} = f_m(M_n) + \alpha_{n,m}, $$
where $\alpha_{n,m}$ is deterministic and can take the values $\pm 1$ or $0$ depending on whether $m$ is even or odd and depending on which classical matrix group $M_n$ is drawn from.   

We now verify condition \eqref{cond:main:3} of Theorem \ref{thm:main}.  Let $m$ be a positive integer.  For any $\eta > 0$ and all $n$ sufficiently large, by Markov's inequality, we have
\begin{align*}
	\Prob \left( \frac{1}{n} \left| \tr M_n^m \right| > \eta \right) &= \Prob \left( \left| f_m(M_n) + \alpha_{n,m} \right| > n \eta \right) \\
		&= \Prob \left( \left| \xi_m^{(n)} + \alpha_{n,m} \right| > n \eta \right) \\
		&\leq \Prob \left( \left| \xi_m^{(n)} \right| > \frac{n \eta}{2} \right) \\
		&\leq  2 \frac{ \E | \xi_m^{(n)} |}{n \eta} \\
		&\leq 2 \frac{\E | \xi_m^{(n)} - \sqrt{m} Z_m |}{n \eta} + 2 \sqrt{m} \frac{\E|Z_m|}{n \eta}.
\end{align*} 
Therefore, by \eqref{eq:CLT:exp}, we conclude that, for any $\eta > 0$, 
$$ \lim_{n \to \infty} \Prob \left( \frac{1}{n} \left| \tr M_n^m \right| > \eta \right) = 0, $$
which completes the verification of condition \eqref{cond:main:3}.  (Alternatively, condition \eqref{cond:main:3} also follows from the results in \cite{MM}.)  

It remains to verify conditions \eqref{cond:main:1} and \eqref{cond:main:2} of Theorem \ref{thm:main}.  Notice that condition \eqref{cond:main:1} follows from condition \eqref{cond:main:2} in the case that $z=0$.  Thus, it suffices to prove condition \eqref{cond:main:2} for all $z \in \mathbb{D}$.  

To this end, define the event
$$ E_n := \left\{ \sqrt{ \sum_{m=0}^N \left| \xi_{m+1}^{(n)} - \sqrt{m+1} Z_{m+1} \right|^2 } \leq \frac{1}{100 \log^2 n} \right\}. $$
By Markov's inequality and \eqref{eq:CLT:exp}, it follows that 
\begin{equation} \label{eq:EnC}
	\lim_{n \to \infty} \Prob \left(E_n^C \right) = 0. 
\end{equation}

It remains to show that, for all $z \in \mathbb{D}$, 
$$ \lim_{\delta \searrow 0} \limsup_{n \to \infty} \Prob \left( \left| \sum_{m=0}^N z^m \overline{\tr M_n^{m+1} } \right| \leq \delta \right) = 0. $$
By symmetry, it suffices to show that for all $z \in \mathbb{D}$, 
$$ \lim_{\delta \searrow 0} \limsup_{n \to \infty} \Prob \left( \left| \sum_{m=0}^N z^m \tr M_n^{m+1} \right| \leq \delta \right) = 0. $$

Fix $z \in \mathbb{D}$.  Observe that
\begin{align*}
	\Prob \left( \left| \sum_{m=0}^N z^m \tr M_n^{m+1} \right| \leq \delta \right) &= \Prob \left( \left| \sum_{m=0}^N z^m ( \xi_{m+1}^{(n)} + \alpha_{n,m+1} ) \right) \leq \delta \right) \\
		&\leq \Prob \left( \left\{ \left| \sum_{m=0}^N z^m ( \xi_{m+1}^{(n)} + \alpha_{n,m+1} ) \right) \leq \delta \right\} \bigcap E_n \right) + \Prob(E_n^C). 
\end{align*}
Thus, by \eqref{eq:EnC}, it suffices to show that
\begin{equation} \label{eq:suffice:pw}
	\lim_{\delta \searrow 0} \limsup_{n \to \infty} \Prob \left( \left\{ \left| \sum_{m=0}^N z^m ( \xi_{m+1}^{(n)} + \alpha_{n,m+1} ) \right) \leq \delta \right\} \bigcap E_n \right) = 0. 
\end{equation} 

Notice that on the event $E_n$, we have
\begin{align*}
	&\left| \sum_{m=0}^N z^m (\xi_{m+1}^{(n)} + \alpha_{n,m+1}) - \sum_{m=0}^N z^m ( \sqrt{m+1} Z_{m+1} + \alpha_{n,m+1} ) \right| \\
	&\qquad\qquad \leq \sum_{m=0}^N \left| \xi_{m+1}^{(n)} - \sqrt{m+1} Z_{m+1} \right| \\
	&\qquad\qquad \leq \sqrt{N+1} \sqrt{ \sum_{m=0}^N \left| \xi_{m+1}^{(n)} - \sqrt{m+1} Z_{m+1} \right|^2 } \\
	&\qquad\qquad \leq \frac{1}{\log n}
\end{align*}
by the Cauchy--Schwarz inequality.  

Thus, by considering just the real part, we conclude that
\begin{align*}
	\Prob &\left( \left\{ \left| \sum_{m=0}^N z^m ( \xi_{m+1}^{(n)} + \alpha_{n,m+1} ) \right) \leq \delta \right\} \bigcap E_n \right) \\
	&\leq \Prob \left( \left| \sum_{m=0}^N z^m (\sqrt{m+1} Z_{m+1} + \alpha_{n,m+1} ) \right| \leq \delta + \frac{1}{\log n} \right) \\
	&\leq \Prob \left( \left| \sum_{m=0}^N \sqrt{m+1} \left(\Re(z^m) \Re(Z_{m+1}) - \Im(z^m) \Im(Z_{m+1})\right) + \Re(z^m)\alpha_{n,m+1} \right| \leq \delta + \frac{1}{\log n} \right) \\
	&\leq \sup_{x \in \mathbb{R}} \Prob \left( \left| \sum_{m=0}^N \sqrt{m+1} \left(\Re(z^m) \Re(Z_{m+1}) - \Im(z^m) \Im(Z_{m+1}) \right) + x \right| \leq \delta + \frac{1}{\log n} \right).
\end{align*}
We now consider two cases.  In the orthogonal, special orthogonal, or symplectic cases, we observe that $\Im(Z_{m+1}) = 0$ for $m = 0, \ldots, N$.  In this case, we have
\begin{align*}
	\sup_{x \in \mathbb{R}} &\Prob \left( \left| \sum_{m=0}^N \sqrt{m+1} \left(\Re(z^m) \Re(Z_{m+1}) - \Im(z^m) \Im(Z_{m+1}) \right) + x \right| \leq \delta + \frac{1}{\log n} \right) \\
		&= \sup_{x \in \mathbb{R}} \Prob \left( \left| \sum_{m=0}^N \sqrt{m+1} \Re(z^m) Z_{m+1} + x \right| \leq \delta + \frac{1}{\log n} \right) \\
		&= \sup_{x \in \mathbb{R}} \Prob \left( \left| \sigma_z Z_1 + x \right| \leq \delta + \frac{1}{\log n} \right),
\end{align*}
where 
$$ \sigma_z^2 := \sum_{m=0}^N |\Re(z^m)|^2 (m+1) \geq 1. $$
Here we used that $Z_1, \ldots, Z_{N+1}$ are iid real standard normal random variables, and hence any linear combination of $Z_1, \ldots, Z_{N+1}$ is also normal.  Thus, we conclude that 
\begin{align*} 
	\sup_{x \in \mathbb{R}} &\Prob \left( \left| \sum_{m=0}^N \sqrt{m+1} \left(\Re(z^m) \Re(Z_{m+1}) - \Im(z^m) \Im(Z_{m+1}) \right) + x \right| \leq \delta + \frac{1}{\log n} \right) \\
	&\leq \sup_{x \in \mathbb{R}} \Prob \left( \left| Z_1 + x \right| \leq \delta + \frac{1}{\log n} \right) \\
	&\leq \sup_{x \in \mathbb{R}} \Prob \left( \left| Z_1 + x \right| \leq \sqrt{2} \delta + \frac{\sqrt{2}}{\log n} \right).
\end{align*}  

For the unitary case, we observe that
\begin{align*}
	\sup_{x \in \mathbb{R}} &\Prob \left( \left| \sum_{m=0}^N \sqrt{m+1} \left(\Re(z^m) \Re(Z_{m+1}) - \Im(z^m) \Im(Z_{m+1}) \right) + x \right| \leq \delta + \frac{1}{\log n} \right) \\
	&\leq \sup_{x \in \mathbb{R}} \Prob \left( \left| \sigma_z' \Re(Z_1) + x \right| \leq \delta + \frac{1}{\log n} \right),
\end{align*}
where 
$$ \sigma_z'^2 := \sum_{m=0}^N |z|^{2m} (m+1) \geq 1. $$
Thus, by the same reasoning as in the other cases, we conclude that
\begin{align*}
	\sup_{x \in \mathbb{R}} &\Prob \left( \left| \sum_{m=0}^N \sqrt{m+1} \left(\Re(z^m) \Re(Z_{m+1}) - \Im(z^m) \Im(Z_{m+1}) \right) + x \right| \leq \delta + \frac{1}{\log n} \right) \\
	&\leq \sup_{x \in \mathbb{R}} \Prob \left( \left| \Re(Z_1) + x \right| \leq \delta + \frac{1}{\log n} \right) \\
	&= \sup_{x \in \mathbb{R}} \Prob \left( \left| Z' + x \right| \leq \sqrt{2}\delta + \frac{\sqrt{2}}{\log n} \right),
\end{align*}
where $Z'$ is a real standard normal random variable.  

Hence, in either case, we obtain
\begin{align*}
	\lim_{\delta \searrow 0} \limsup_{n \to \infty} \Prob &\left( \left\{ \left| \sum_{m=0}^N z^m ( \xi_{m+1}^{(n)} + \alpha_{n,m+1} ) \right) \leq \delta \right\} \bigcap E_n \right) \\
	&\leq \lim_{\delta \searrow 0} \limsup_{n \to \infty} \sup_{x \in \mathbb{R}} \Prob \left( \left| Z' + x \right| \leq \sqrt{2} \delta + \frac{\sqrt{2}}{\log n} \right) \\
	&\leq \lim_{\delta \searrow 0} \sup_{x \in \mathbb{R}} \Prob \left( \left| Z' + x \right| \leq \delta \right) = 0
\end{align*}
by a simple calculation involving the density of the standard normal distribution.  This verifies \eqref{eq:suffice:pw}, and the proof of Theorem \ref{thm:compact} is complete.

\subsection{Proof of Corollary \ref{cor:dependent}}

Let $\theta_1, \theta_2, \ldots$ be a sequence of iid random variables distributed uniformly on $[0,2 \pi)$.  For each $n \geq 1$, set
$$ p_{2n}(z) := \prod_{j=1}^{n} (z - e^{i \theta_j})(z - e^{- i \theta_j}) $$
and
$$ p_{2n - 1}(z) := (z - e^{i \theta_{n}}) \prod_{j=1}^{n-1} (z - e^{i \theta_j})(z - e^{- i \theta_j}). $$
We will apply Theorem \ref{thm:alternative} to show that $\mu_{p_n}'$ converges in probability to the uniform probability distribution on the unit circle centered at the origin as $n \to \infty$.  From this, the conclusion of Corollary \ref{cor:dependent} follows immediately.  

Define the triangular array $(\theta_j^{(n)})_{j \leq n}$ of random variables on $[0,2\pi)$ by 
$$ \theta_j^{(n)} := \left\{
	\begin{array}{ll}
		2 \pi - \theta_{j/2},& \text{ if } j \text{ even},\\
		\theta_{ (j+1)/2},& \text{ if } j \text { odd}.
	\end{array} \right. $$
It follows that, for all $n \geq 1$, 
$$ p_n(z) = \prod_{j=1}^n (z - e^{i \theta_j^{(n)}}). $$
Thus, it remains to show that the triangular array $(\theta_j^{(n)})_{j \leq n}$ satisfies the three assumptions of Theorem \ref{thm:alternative}.  

We begin by verifying condition \eqref{cond:alt:3} of Theorem \ref{thm:alternative}.  We observe that, for any integer $m \geq 1$, 
$$ \frac{1}{n} \sum_{j=1}^n e^{i \theta_j^{(n)} m} = \frac{1}{n} \sum_{\substack{ 1 \leq j \leq n \\ j \text{ even }} } e^{i \theta_j^{(n)} m} + \frac{1}{n} \sum_{\substack{ 1 \leq j \leq n \\ j \text{ odd} }} e^{i \theta_j^{(n)} m}. $$ 
Since both sums on the right-hand side are sums of iid random variables, we apply the law of large numbers twice to obtain
$$ \frac{1}{n} \sum_{j=1}^n e^{i \theta_j^{(n)} m} \longrightarrow 0 $$
almost surely as $n \to \infty$.  

Conditions \eqref{cond:alt:1} and \eqref{cond:alt:2} of Theorem \ref{thm:alternative} will follow from Lemma \ref{lemma:dependentsum} below.

\begin{lemma} \label{lemma:dependentsum} 
Let $f: [0,2\pi) \to \mathbb{R}$ be a function such that $f(\theta_1) + f(2 \pi -\theta_1)$ is non-degenerate.  Then, for any $\delta > 0$, 
$$ \limsup_{n \to \infty} \Prob \left( \left| \sum_{j=1}^n f( \theta_j^{(n)} ) \right| \leq \delta \right) = 0. $$
\end{lemma}
\begin{proof}
We consider two cases.  First, if $n$ is even, by \cite[Theorem 2.22]{P}, we have
\begin{align*}
	\Prob \left( \left| \sum_{j=1}^n f( \theta_j^{(n)} ) \right| \leq \delta \right) &\leq \Prob \left( \left| \sum_{j=1}^{n/2} ( f( \theta_j ) + f(2\pi - \theta_j)) \right| \leq \delta \right) \\
		&\leq \sup_{x \in \mathbb{R}} \Prob \left( x \leq \sum_{j=1}^{n/2} (f( \theta_j )+ f(2\pi-\theta_j)) \leq x + 2 \delta \right) \\
		&\leq C_f \frac{1 + 2 \delta}{n^{1/2}},
\end{align*}
where $C_f > 0$ is a constant that only depends on $f$.  In the case that $n > 1$ is odd, by \cite[Lemma 1.11]{P} and \cite[Theorem 2.22]{P}, we obtain
\begin{align*}
	\Prob \left( \left| \sum_{j=1}^n f( \theta_j^{(n)} ) \right| \leq \delta \right) &\leq \Prob \left( \left| \sum_{j=1}^{(n-1)/2} ( f( \theta_j ) + f(2\pi - \theta_j)) + f(\theta_n^{(n)}) \right| \leq \delta \right) \\
		&\leq \sup_{x \in \mathbb{R}} \Prob \left( x \leq \sum_{j=1}^{(n-1)/2} (f( \theta_j )+ f(2\pi-\theta_j)) + f(\theta_n^{(n)}) \leq x + 2 \delta \right) \\
		&\leq \sup_{x \in \mathbb{R}} \Prob \left( x \leq \sum_{j=1}^{(n-1)/2} (f( \theta_j )+ f(2\pi-\theta_j)) \leq x + 2 \delta \right) \\
		&\leq C_f \frac{1 + 2 \delta}{(n-1)^{1/2}},
\end{align*}
and the proof is complete.  
\end{proof}

To verify condition \eqref{cond:alt:1} of Theorem \ref{thm:alternative}, we note that, for any $\delta > 0$, 
\begin{align*}
	\Prob \left( \left| \sum_{j=1}^n e^{- i\theta_j^{(n)}} \right| \leq \delta \right) \leq \Prob \left( \left| \sum_{j=1}^n \cos (\theta_j^{(n)} ) \right| \leq \delta \right).
\end{align*}
Thus, by Lemma \ref{lemma:dependentsum}, we conclude that, for any $\delta > 0$, 
$$ \limsup_{n \to \infty} \Prob \left( \left| \sum_{j=1}^n e^{- i\theta_j^{(n)}} \right| \leq \delta \right) = 0. $$

It remains to verify condition \eqref{cond:alt:2} of Theorem \ref{thm:alternative}.  To this end, fix $z \in \mathbb{D}$, and let $\eta > 0$.  As 
$$ \left| \sum_{j=1}^n \frac{1}{z - e^{- i\theta_j^{(n)}}} \right| \leq \sum_{j=1}^n \left| \frac{1}{ z - e^{-i \theta_j^{(n)}}} \right| \leq \frac{n}{1 - |z|}, $$
it follows that
$$ \lim_{n \to \infty} \Prob \left( \frac{1}{n} \log \left| \sum_{j=1}^n \frac{1}{z - e^{- i\theta_j^{(n)}}} \right| > \eta \right) = 0. $$
On the other hand,
\begin{align*}
	\Prob \left( \frac{1}{n} \log \left| \sum_{j=1}^n \frac{1}{z - e^{- i\theta_j^{(n)}}} \right| < -\eta \right) &\leq \Prob \left( \left| \sum_{j=1}^n \frac{1}{z - e^{- i\theta_j^{(n)}}} \right| < e^{- n \eta} \right) \\
		&\leq \Prob \left( \left| \sum_{j=1}^n \frac{ \Re(z) - \cos(\theta_j^{(n)})}{\left|z - e^{-i \theta_j^{(n)}} \right|^2} \right| < e^{- n \eta} \right), 
\end{align*}
and hence 
$$ \lim_{n \to \infty} \Prob \left( \frac{1}{n} \log \left| \sum_{j=1}^n \frac{1}{z - e^{- i\theta_j^{(n)}}} \right| < -\eta \right) = 0 $$
by Lemma \ref{lemma:dependentsum}.  Therefore, we conclude that 
$$ \frac{1}{n} \log \left| \sum_{j=1}^n \frac{1}{z - e^{- i\theta_j^{(n)}}} \right| \longrightarrow 0 $$
in probability as $n \to \infty$, and the proof of Corollary \ref{cor:dependent} is complete.

\section{Proof of Theorems \ref{thm:main} and \ref{thm:alternative}} \label{sec:proof}

This section is devoted to the proof of Theorems \ref{thm:main} and \ref{thm:alternative}.  

\subsection{Convergence of radial components implies convergence of the empirical measures}
Both Theorems \ref{thm:main} and \ref{thm:alternative} will follow from Lemma \ref{lemma:radii} below.

\begin{lemma}[Convergence of radial components implies convergence of measures] \label{lemma:radii}
For each $n \geq 1$, let $\theta_1^{(n)}, \ldots, \theta_n^{(n)}$ be random variables on $[0,2\pi)$, and set 
$$ p_n(z) := \prod_{j=1}^n (z - e^{i \theta_j^{(n)}}). $$
Let $r_1^{(n)} e^{i \phi_1^{(n)}}, \ldots, r_{n-1}^{(n)} e^{i \phi_{n-1}^{(n)}}$ be the zeros of $p_n'$ in polar form.  Assume
\begin{enumerate}[(i)]
\item for all integers $m \geq 1$, 
\begin{equation} \label{eq:moments}
	\frac{1}{n} \sum_{j=1}^n e^{i m \theta_j^{(n)}} \longrightarrow 0 
\end{equation}
in probability as $n \to \infty$,
\item \label{item:radii:2} we have
\begin{equation} \label{eq:radii}
	\frac{1}{n-1} \sum_{j=1}^{n-1} (1 - r_j^{(n)}) \longrightarrow 0 
\end{equation}
in probability as $n \to \infty$.   
\end{enumerate}
Then $\mu_{p_n}'$ converges in probability as $n \to \infty$ to the uniform probability distribution on the unit circle centered at the origin.  
\end{lemma}

\begin{remark}
By the Gauss--Lucas theorem (Theorem \ref{thm:gauss}), it follows that 
$$ \sup_{n \geq 1} \max_{1 \leq j \leq n} r_j^{(n)} \leq 1. $$  
Thus, condition \eqref{item:radii:2} of Lemma \ref{lemma:radii} implies that most of the roots of $p_n'$ are close to the unit circle centered at the origin.  
\end{remark}

The remainder of this subsection will be devoted to proving Lemma \ref{lemma:radii}.  In particular, we will need the following result, which is adapted from \cite[Proposition 3.2]{S}.  

\begin{lemma} \label{lemma:moments}
Let $n \geq 2$.  Let $x_1, \ldots, x_n \in \mathbb{C}$ with $|x_j| \leq \tau$ for all $1 \leq j \leq n$.  Let $y_1, \ldots, y_{n-1}$ be the critical points of $p(z) := \prod_{j=1}^n (z-x_j)$.  Then, for any integer $k \geq 1$, there exists a constant $C > 0$ (depending only on $\tau$ and $k$) such that
$$ \left| \frac{1}{n} \sum_{j=1}^n x_j^k - \frac{1}{n-1} \sum_{j=1}^{n-1} y_j^k \right| \leq \frac{C}{n-1}. $$
\end{lemma}

In order to prove Lemma \ref{lemma:moments}, we will need the following result from \cite{CN}. 

\begin{lemma} \label{lemma:companion}
Let $n \geq 2$.  If $x_1, \ldots, x_n \in \mathbb{C}$ are the roots of $p(z) := \prod_{j=1}^n (z - x_j)$, and $p$ has critical points $y_1, \ldots, y_{n-1}$, then the matrix
$$ D \left(I_{n-1} - \frac{1}{n} J \right) + \frac{x_n}{n} J $$
has $y_1, \ldots, y_{n-1}$ as its eigenvalues, where $D = \diag(x_1, \ldots,x_{n-1})$, $I_{n-1}$ is the identity matrix of order $n-1$, and $J$ is the $(n-1) \times (n-1)$ matrix of all entries $1$.  
\end{lemma}

We now prove Lemma \ref{lemma:moments}.

\begin{proof}[Proof of Lemma \ref{lemma:moments}]
The proof presented here is adapted from the proof given in \cite{S}.  We observe that it suffices to show
$$ \left| \sum_{j=1}^{n-1} x_j^k - \sum_{j=1}^{n-1} y_j^k \right| \leq C, $$
where $C > 0$ depends only on $\tau$ and $k$.

Let $D = \diag(x_1, \ldots,x_{n-1})$.  Then, by Lemma \ref{lemma:companion}, it follows that
$$ \sum_{j=1}^{n-1} y_j^k = \tr \left( D - \frac{1}{n} DJ - \frac{x_n}{n} J \right)^k, $$
where $J$ is the $(n-1) \times (n-1)$ matrix of all entries $1$.  Thus, it suffices to show
\begin{equation} \label{eq:traceshow}
	\left| \tr \left( D - \frac{1}{n} DJ - \frac{x_n}{n} J \right)^k - \tr D^k \right| \leq C. 
\end{equation}

We note that $\left( D - \frac{1}{n} DJ - \frac{x_n}{n} J \right)^k$ can be written as the sum over all terms of the form
\begin{equation} \label{eq:expandterm}
	D^{l_1} \left( - \frac{1}{n} D J \right)^{l_2} \left( \frac{x_n}{n} J \right)^{l_3} \cdots D^{l_{3k-2}} \left( - \frac{1}{n} D J \right)^{l_{3k - 1}} \left( \frac{x_n}{n} J \right)^{l_{3k}}, 
\end{equation}
where $l_1, \ldots, l_{3k}$ are non-negative integers such that $l_{3j-2} + l_{3j-1} + l_{3j} = 1$ for each $1 \leq j \leq k$.  The total number of such terms is $3^k$.  One of the terms is $D^k$.  We will show that the each of the remaining $3^k-1$ terms can be uniformly bounded by a constant which only depends on $\tau$ and $k$.  

Fix $l_1, \ldots, l_{3k}$ such that the term given in \eqref{eq:expandterm} is not $D^k$.  In order to simplify the expression in \eqref{eq:expandterm}, we observe that
$$ J^m = (n-1)^{m-1} J $$
for all $m \geq 1$.  We also have 
$$ (D^p J) (D^q J) = \left( \sum_{j=1}^{n-1} x_j^q\right) (D^p J), $$
for any $p,q \geq 0$.  

Thus, the term in \eqref{eq:expandterm} can be written as
\begin{equation} \label{eq:expandterm2}
	(-1)^p x_n^q \left(\frac{n-1}{n} \right)^{s_0} \left( \frac{ \sum_{j=1}^{n-1} x_j }{n} \right)^{s_1} \cdots \left( \frac{ \sum_{j=1}^{n-1} x_j^{k-1} }{n} \right)^{s_{k-1}} M, 
\end{equation}
where $p,q,s_0, \ldots, s_{k-1}$ are non-negative integers no larger than $k$, and $M$ is a $(n-1) \times (n-1)$ matrix.  In particular, $M$ is of the form $\frac{1}{n} D^m J$ or $\frac{1}{n}D^{m_1} J D^{m_2}$ for some non-negative integers $m, m_1, m_2$ which are no larger than $k$.  

The scalar term in \eqref{eq:expandterm2} can be uniformly bounded by a constant depending only on $\tau$ and $k$ since $\max_{1 \leq j \leq n} |x_j| \leq \tau$.  If $M = \frac{1}{n} D^m J$, then
$$ \left| \tr (M) \right| = \frac{1}{n} \left| \sum_{j=1}^{n-1} x_j^m \right| \leq \frac{n-1}{n} \tau^m \leq \tau^k $$
since $m \leq k$.  Similarly, if $M = \frac{1}{n}D^{m_1} J D^{m_2}$, then
$$ \left| \tr(M) \right| = \frac{1}{n} \left| \tr(D^{m_1 + m_2} J) \right| = \frac{1}{n} \left| \sum_{j=1}^{n-1} x_j^{m_1 + m_2} \right| \leq \tau^{2k} $$
because $m_1,m_2 \leq k$.  

Combining the bounds above yields \eqref{eq:traceshow}, and the proof is complete.  
\end{proof}

With Lemma \ref{lemma:moments} in hand, we can now prove Lemma \ref{lemma:radii}.  

\begin{proof}[Proof of Lemma \ref{lemma:radii}]
By \eqref{eq:moments} and Lemma \ref{lemma:moments}, it follows that, for each $m \geq 1$, 
$$ \frac{1}{n-1} \sum_{j=1}^{n-1} \left( r_j^{(n)} \right)^m e^{i \phi_j^{(n)} m} \longrightarrow 0 $$
in probability as $n \to \infty$.  By the Gauss--Lucas theorem (Theorem \ref{thm:gauss}), 
$$ \sup_{n \geq 1} \max_{1 \leq j \leq n } r_j^{(n)} \leq 1. $$ 
Thus,
\begin{align*}
	\frac{1}{n-1} \left| \sum_{j=1}^{n-1} \left( r_j^{(n)} \right)^m e^{i \phi_j^{(n)} m} - \sum_{j=1}^{n-1} e^{i \phi_j^{(n)} m} \right| &\leq \frac{1}{n-1} \sum_{j=1}^{n-1} \left| 1 - \left( r_j^{(n)} \right)^m \right| \\
		&\leq \frac{C_m}{n-1} \sum_{j=1}^{n-1} \left( 1 - r_j^{(n)} \right),
\end{align*}
where $C_m > 0$ depends only on $m$.  Hence, by \eqref{eq:radii}, we conclude that, for any $m \geq 1$, 
$$ \frac{1}{n-1} \sum_{j=1}^{n-1} e^{i m \phi_j^{(n)}} \longrightarrow 0 $$
in probability as $n \to \infty$.  This also implies that, for any $m \geq 1$,
$$ \frac{1}{n-1} \sum_{j=1}^{n-1} e^{-im \phi_j^{(n)}} \longrightarrow 0 $$
in probability as $n \to \infty$.  In other words, for any trigonometric polynomial $q$, 
\begin{equation} \label{eq:trigpoly}
	\frac{1}{n-1} \sum_{j=1}^{n-1} q(\phi_j^{(n)}) \longrightarrow \E[q(\xi)]
\end{equation}
in probability as $n \to \infty$, where $\xi$ is a random variable uniformly distributed on $[0,2\pi)$.  

Let $f:\mathbb{C} \to \mathbb{R}$ be a bounded Lipschitz continuous function.  By the Portemanteau theorem (see, for example, \cite[Theorem 13.16]{Kprob}), it suffices to show that 
$$ \frac{1}{n-1} \sum_{j=1}^{n-1} f \left( r_j^{(n)} e^{i \phi_j^{(n)}} \right) \longrightarrow \E[ f(e^{i \xi}) ] $$
in probability as $n \to \infty$.  

Let $\eps > 0$.  By \cite[Theorem 4.25]{R}, there exists a trigonometric polynomial $q$ such that 
\begin{equation} \label{eq:trigapprox}
	\sup_{t \in [0,2\pi]} |f(e^{it}) - q(t)| \leq \eps. 
\end{equation}
Then, by the triangle inequality, we have
\begin{align*}
	\left| \frac{1}{n-1} \sum_{j=1}^{n-1} f \left( r_j^{(n)} e^{i \phi_j^{(n)}} \right) - \E[ f(e^{i\xi}) ] \right| &\leq \frac{1}{n-1}\left| \sum_{j=1}^{n-1} f \left( r_j^{(n)} e^{i \phi_j^{(n)}} \right) -  \sum_{j=1}^{n-1} f \left( e^{i \phi_j^{(n)}} \right) \right| \\
	&\qquad + \frac{1}{n-1}\left| \sum_{j=1}^{n-1} f \left( e^{i \phi_j^{(n)}} \right) - \sum_{j=1}^{n-1} q \left( { \phi_j^{(n)}} \right) \right| \\
	&\qquad + \left| \frac{1}{n-1} \sum_{j=1}^{n-1} q \left( { \phi_j^{(n)}} \right) - \E[q(\xi)] \right| \\
	&\qquad + \left| \E[q(\xi)] - \E[f(e^{i \xi})] \right|.
\end{align*}
Since $f$ is Lipschitz continuous, we obtain
$$ \frac{1}{n-1}\left| \sum_{j=1}^{n-1} f \left( r_j^{(n)} e^{i \phi_j^{(n)}} \right) -  \sum_{j=1}^{n-1} f \left( e^{i \phi_j^{(n)}} \right) \right| \leq \frac{C_f}{n-1} \sum_{j=1}^{n-1} \left( 1- r_j^{(n)} \right), $$
where $C_f$ is the Lipschitz constant of $f$.  By \eqref{eq:trigapprox}, we have
$$ \frac{1}{n-1}\left| \sum_{j=1}^{n-1} f \left( e^{i \phi_j^{(n)}} \right) - \sum_{j=1}^{n-1} q \left( { \phi_j^{(n)}} \right) \right| \leq \eps $$
and
$$ \left| \E[q(\xi)] - \E[f(e^{i\xi})] \right| \leq \eps. $$
Thus, we conclude that
\begin{align*}
	&\left| \frac{1}{n-1} \sum_{j=1}^{n-1} f \left( r_j^{(n)} e^{i \phi_j^{(n)}} \right) - \E[ f(e^{i\xi}) ] \right| \\
		&\qquad\qquad\qquad \leq \frac{C_f}{n-1} \sum_{j=1}^{n-1} \left( 1- r_j^{(n)} \right) + \left| \frac{1}{n-1} \sum_{j=1}^{n-1} q \left( { \phi_j^{(n)}} \right) - \E[q(\xi)] \right| + 2\eps.
\end{align*}
The claim now follows from \eqref{eq:radii} and \eqref{eq:trigpoly}.  
\end{proof}

\subsection{Convergence of the radial components}

In order to apply Lemma \ref{lemma:radii}, we must verify the convergence in \eqref{eq:radii}.  We do so in the following lemmata.  

\begin{lemma} \label{lemma:main}
For each $n \geq 1$, let $\theta_1^{(n)}, \ldots, \theta_n^{(n)}$ be random variables on $[0,2\pi)$, and set 
$$ p_n(z) := \prod_{j=1}^n (z - e^{i \theta_j^{(n)}}). $$
Let $\zeta_1^{(n)}, \ldots, \zeta_{n-1}^{(n)}$ be the zeros of $p_n'$.  Assume
\begin{enumerate}[(i)]
\item we have
$$ \lim_{\delta \searrow 0} \limsup_{n \to \infty} \Prob \left( \left| \sum_{j=1}^n e^{-i \theta_j^{(n)}} \right| \leq \delta \right) = 0, $$
\item for almost every $z \in \mathbb{D}$,
$$ \lim_{\delta \searrow 0} \limsup_{n \to \infty} \Prob \left( \left| \sum_{m=0}^{\lfloor \log^2 n \rfloor} z^m \sum_{j=1}^n e^{-i \theta_j^{(n)} (m+1)} \right| \leq \delta \right) = 0. $$
\end{enumerate}
Then, for any $0 < \eps < 1$ and for every infinitely differentiable function $\varphi:\mathbb{C} \to \mathbb{R}$ supported on $\mathbb{D}_{1-\eps}$,
$$ \frac{1}{n-1} \sum_{j=1}^{n-1} \varphi( \zeta_{j}^{(n)} ) \longrightarrow 0 $$
in probability as $n \to \infty$.  
\end{lemma}

We also have the following alternative formulation of Lemma \ref{lemma:main}.

\begin{lemma}[Alternative formulation] \label{lemma:alternative}
For each $n \geq 1$, let $\theta_1^{(n)}, \ldots, \theta_n^{(n)}$ be random variables on $[0,2\pi)$, and set 
$$ p_n(z) := \prod_{j=1}^n (z - e^{i \theta_j^{(n)}}). $$
Let $\zeta_1^{(n)}, \ldots, \zeta_{n-1}^{(n)}$ be the zeros of $p_n'$.  Assume
\begin{enumerate}[(i)]
\item we have
$$ \lim_{\delta \searrow 0} \limsup_{n \to \infty} \Prob \left( \left| \sum_{j=1}^n e^{-i \theta_j^{(n)}} \right| \leq \delta \right) = 0, $$
\item for almost every $z \in \mathbb{D}$,
$$ \frac{1}{n} \log \left| \sum_{j=1}^n \frac{1}{z - e^{i \theta_j^{(n)}}} \right| \longrightarrow 0 $$
in probability as $n \to \infty$.
\end{enumerate}
Then, for any $0 < \eps < 1$ and for every infinitely differentiable function $\varphi:\mathbb{C} \to \mathbb{R}$ supported on $\mathbb{D}_{1-\eps}$,
$$ \frac{1}{n-1} \sum_{j=1}^{n-1} \varphi( \zeta_{j}^{(n)} ) \longrightarrow 0 $$
in probability as $n \to \infty$.  
\end{lemma}

We will prove Lemmas \ref{lemma:main} and \ref{lemma:alternative} in Section \ref{sec:lemmas}.  We now complete the proof of Theorems \ref{thm:main} and \ref{thm:alternative} assuming Lemmas \ref{lemma:main} and \ref{lemma:alternative}.  We prove both theorems simultaneously.  

\begin{proof}[Proof of Theorems \ref{thm:main} and \ref{thm:alternative}]
Let $\zeta_1^{(n)}, \ldots, \zeta_{n-1}^{(n)}$ be the zeros of $p_n'$.  In view of Lemma \ref{lemma:radii}, it suffices to show that
$$ \frac{1}{n-1} \sum_{j=1}^{n-1} (1 - |\zeta_j^{(n)}|) \longrightarrow 0 $$
in probability as $n \to \infty$.  

Let $0 < \eps < 1/2$.  Let $\varphi:\mathbb{C} \to [0,1]$ be an infinitely differentiable function such that $\varphi$ takes the value $1$ on $\mathbb{D}_{1 - 2\eps}$ and takes the value zero on $\mathbb{C} \setminus \mathbb{D}_{1 - \eps}$.  By Lemma \ref{lemma:main} (alternatively, Lemma \ref{lemma:alternative}), we have
\begin{equation} \label{eq:phiprob}
	\frac{1}{n-1} \sum_{j=1}^{n-1} \varphi( \zeta_j^{(n)} ) \longrightarrow 0
\end{equation} 
in probability as $n \to \infty$.  

On the other hand, by the Gauss--Lucas theorem (Theorem \ref{thm:gauss}), it follows that 
$$ \sup_{n \geq 1} \max_{1 \leq j \leq n} |\zeta_j^{(n)}| \leq 1. $$
Thus, we have
\begin{align*}
	\frac{1}{n-1} \sum_{j=1}^{n-1} (1 - |\zeta_j^{(n)}| ) &= \frac{1}{n-1} \sum_{j=1}^{n-1} (1 - |\zeta_j^{(n)}|) \indicator{ \zeta_j^{(n)} \in \mathbb{D}_{1 - 2 \eps}} \\
		&\qquad\qquad + \frac{1}{n-1} \sum_{j=1}^{n-1} (1 - |\zeta_j^{(n)}|) \indicator{ \zeta_j^{(n)} \notin \mathbb{D}_{1 - 2 \eps}} \\
		&\leq \frac{1}{n-1} \sum_{j=1}^{n-1} \indicator{ \zeta_j^{(n)} \in \mathbb{D}_{1 - 2 \eps}} + 2 \eps \\
		&\leq \frac{1}{n-1} \sum_{j=1}^{n-1} \varphi(\zeta_j^{(n)}) + 2 \eps.
\end{align*}
Since $\eps$ was arbitrary, the claim now follows from \eqref{eq:phiprob}.  
\end{proof}

\section{Proof of Lemmas \ref{lemma:main} and \ref{lemma:alternative}} \label{sec:lemmas}

It remains to verify Lemmas \ref{lemma:main} and \ref{lemma:alternative}.  The proof is based on a connection with logarithmic potential theory.  In particular, we will exploit the following formula from \cite[Section 2.4.1]{HKPV}: for every analytic function $f$ which does not vanish identically, 
\begin{equation} \label{eq:log}
	\frac{1}{2 \pi} \Delta \log |f| = \sum_{z \in \mathbb{C} : f(z) = 0} \mathcal{N}_f(z) \delta_z,
\end{equation}
where $\mathcal{N}_f(z)$ is the multiplicity of the zero at $z$ and $\delta_z$ is the unit point mass at $z$.  Here $\Delta$ is the Laplace operator, which should be interpreted in the distributional sense.  Similar methods also appeared in \cite{K,KZ,TV}.  In fact, our overall strategy is based on the arguments presented in \cite{K}.  

Let $\theta_1^{(n)}, \ldots, \theta_n^{(n)}$, $\zeta_1^{(n)}, \ldots, \zeta_{n-1}^{(n)}$, and $p_n$ be as in Lemma \ref{lemma:main} (alternatively, Lemma \ref{lemma:alternative}).  Consider the logarithmic derivative of $p_n$:
\begin{equation} \label{eq:defLn}
	L_n(z) := \frac{p_n'(z)}{p_n(z)} = \sum_{j=1}^{n} \frac{1}{z - e^{i \theta_j^{(n)}}}. 
\end{equation}
Let $0 < \eps < 1$, and let $\varphi: \mathbb{C} \to \mathbb{R}$ be an infinitely differentiable function supported on $\mathbb{D}_{1 - \eps}$.  In view of \eqref{eq:log}, we have
$$ \frac{1}{2 \pi n} \int_{\mathbb{C}} \left( \log |L_n(z)| \right) \Delta \varphi(z) d \lambda(z) = \frac{1}{n} \sum_{j=1}^{n-1} \varphi(\zeta_j^{(n)}) - \frac{1}{n} \sum_{j=1}^n \varphi( e^{i \theta_j^{(n)}}), $$
where $\lambda$ denotes Lebesgue measure on $\mathbb{C}$.  Since $\varphi$ is supported on $\mathbb{D}_{1-\eps}$, the above equality becomes
$$ \frac{1}{2 \pi n} \int_{\mathbb{C}} \left( \log |L_n(z)| \right) \Delta \varphi(z) d \lambda(z) = \frac{1}{n} \sum_{j=1}^{n-1} \varphi(\zeta_j^{(n)}). $$
Therefore, the proof of Lemma \ref{lemma:main} (alternatively, Lemma \ref{lemma:alternative}) reduces to showing that
\begin{equation} \label{eq:convshow}
	\frac{1}{n} \int_{\mathbb{C}} \left( \log |L_n(z)| \right) \Delta \varphi(z) d \lambda(z) \longrightarrow 0 
\end{equation}
in probability as $n \to \infty$.  

In order to verify \eqref{eq:convshow}, we will need the following result from \cite{TV}.  

\begin{lemma}[Lemma 3.1 from \cite{TV}] \label{lemma:TVtight}
Let $(X,\mathcal{A}, \nu)$ be a finite measure space.  Let $f_1, f_2, \ldots: X \to \mathbb{R}$ be random functions which are defined over a probability space $(\Omega,\mathcal{B},\Prob)$ and are jointly measurable with respect to $\mathcal{A} \otimes \mathcal{B}$.  Assume that:
\begin{enumerate}[(i)]
\item for $\nu$--almost every $x \in X$, $f_n(x)$ converges in probability to zero as $n \to \infty$,
\item for some $\delta > 0$, the sequence $\int_{X} |f_n(x)|^{1+\delta} d \nu(x)$ is tight.
\end{enumerate}
Then $\int_{X} f_n(x) d \nu(x)$ converges in probability to zero as $n \to \infty$.  
\end{lemma}

In order to apply Lemma \ref{lemma:TVtight}, we will show that $\frac{1}{n} \log |L_n(z)|$ converges in probability to zero for a.e. $z \in \mathbb{D}_{1-\eps}$ and that the sequence $\frac{1}{n^2} \int_{\mathbb{D}_{1-\eps}} \log^2 |L_n(z)| d\lambda(z)$ is tight.  To this end, we define
\begin{align*}
	\log_- x := \left\{
	\begin{array}{lr}
		|\log x|, & 0 \leq x \leq 1\\
		0, & x > 1,
	\end{array} \right. \qquad
	\log_+ x := \left\{
	\begin{array}{lr}
		0 & 0 \leq x \leq 1\\
		\log x, & x > 1.
	\end{array} \right. 
\end{align*}
From \eqref{eq:defLn} it follows that $L_n(z)$ is finite for all $z \in \mathbb{D}$.  Moreover, $L_n(z) = 0$ only when $p'_n(z) = 0$; in this case, $\log_- |L_n(z)| = \infty$.  

\subsection{Pointwise convergence of $L_n(z)$}

This subsection is devoted to the following lemma.  

\begin{lemma} \label{lemma:pointwise}
If, for almost every $z \in \mathbb{D}$,
\begin{equation} \label{eq:conditionpointwise}
	\lim_{\delta \searrow 0} \limsup_{n \to \infty} \Prob \left( \left| \sum_{m=0}^{\lfloor \log^2 n \rfloor} z^m \sum_{j=1}^n e^{-i \theta_j^{(n)} (m+1)} \right| \leq \delta \right) = 0, 
\end{equation}
then, for almost every $z \in \mathbb{D}$,
$$ \frac{1}{n} \log |L_n(z)| \longrightarrow 0 $$
in probability as $n \to \infty$.  
\end{lemma}
\begin{proof}
For $|z| < 1$, we have, by Fubini's theorem, 
\begin{align*}
	L_n(z) &= -\sum_{j=1}^n \frac{1}{e^{i \theta_j^{(n)}}} \frac{1}{1 - \frac{z}{e^{i \theta_j^{(n)}}}} = -\sum_{j=1}^n \frac{1}{e^{i \theta_j^{(n)}}} \sum_{m=0}^\infty \frac{z^m}{e^{i \theta_j^{(n)} m}} = - \sum_{m=0}^\infty z^m \sum_{j=1}^n e^{-i \theta_j^{(n)} (m+1)} \\
		&= -\sum_{m=0}^\infty z^m T_m^{(n)}(z), 
\end{align*}
where
$$ T_m^{(n)}(z) := \sum_{j=1}^n e^{-i \theta_j^{(n)} (m+1)}. $$
Here the use of Fubini's theorem is justified since
$$ \sum_{m=0}^\infty \sum_{j=1}^n |z|^m \left| e^{-i \theta_j^{(n)} (m+1)} \right| \leq n \sum_{m=0}^\infty |z|^m < \infty $$
for all $|z| < 1$ and every $n \geq 1$.  

Let $0 < \eps < 1$, and fix $z \in \mathbb{D}$ with $|z| \leq 1- \eps$ such that \eqref{eq:conditionpointwise} holds.  Set $N := \lfloor \log^2 n \rfloor$.  We can then write
\begin{equation} \label{eq:useful}
	|L_n(z)| = \left| \sum_{m=0}^N z^m T_m^{(n)}(z) + \sum_{m={N+1}}^\infty z^m T_m^{(n)}(z) \right|.
\end{equation}
Since $|T_m^{(n)}(z)| \leq n$ for all integers $m \geq 0$, it follows that 
\begin{equation} \label{eq:Lnlastupper}
	\left| \sum_{m={N+1}}^\infty z^m T_m^{(n)}(z) \right| \leq \frac{n |z|^{N+1}}{1-|z|} \leq \frac{n (1-\eps)^{N+1}}{\eps}. 
\end{equation}
In addition, we observe that
\begin{equation} \label{eq:Lnupper}
	|L_n(z)| \leq \sum_{j=1}^n \left| \frac{1}{z - e^{i \theta_j^{(n)} }} \right| \leq \frac{n}{\eps}.
\end{equation}

Since $\eps$ is arbitrary, it suffices to show that $\frac{1}{n} \log |L_n(z)|$ converges to zero in probability.  Since 
$$ \log |L_n(z)| = \log_+ |L_n(z)| - \log_- |L_n(z)|, $$
it suffices to show that both $\frac{1}{n} \log_+ |L_n(z)|$ and  $\frac{1}{n} \log_- |L_n(z)|$ converge to zero in probability.  

For the first term, we have
$$ 0 \leq \frac{1}{n} \log_+ |L_n(z)| \leq \frac{1}{n} \log \left( \frac{n}{\eps} \right) $$
by \eqref{eq:Lnupper}.  Thus, $\frac{1}{n} \log_+ |L_n(z)|$ converges to zero a.s.  

It now suffices to show that, for any $\eta > 0$, 
$$ \lim_{n \to \infty} \Prob \left( \frac{1}{n} \log_{-} |L_n(z)| > \eta \right) = 0. $$
However, from \eqref{eq:useful} and \eqref{eq:Lnlastupper}, we observe that
\begin{align*}
	\log_- |L_n(z)| &\leq \frac{1}{|L_n(z)|} \\
		&= \frac{1}{ \left| \sum_{m=0}^N z^m T_m^{(n)}(z) + \sum_{m=N+1}^\infty z^m T_m^{(n)}(z) \right|} \\
		&\leq \frac{1}{ \left| \sum_{m=0}^N z^m T_m^{(n)}(z) \right| - \frac{n (1-\eps)^{N+1}}{\eps} }
\end{align*}
provided 
$$ \left| \sum_{m=0}^N z^m T_m^{(n)}(z) \right| > \frac{n (1-\eps)^{N+1}}{\eps}. $$
Thus, we obtain
\begin{align*}
	 \Prob &\left( \frac{1}{n} \log_{-} |L_n(z)| > \eta \right) \\
	 &\qquad \leq \Prob \left( \frac{1}{n} \log_{-} |L_n(z)| > \eta \text{ and } \left| \sum_{m=0}^N z^m T_m^{(n)}(z) \right| > \frac{n (1-\eps)^{N+1}}{\eps} \right) \\
	 &\qquad\qquad + \Prob \left( \left| \sum_{m=0}^N z^m T_m^{(n)}(z) \right| \leq \frac{n (1-\eps)^{N+1}}{\eps} \right) \\
	 &\qquad \leq 2\Prob \left( \left| \sum_{m=0}^N z^m T_m^{(n)}(z) \right| \leq \frac{1}{n \eta} + \frac{n (1-\eps)^{N+1}}{\eps} \right).
\end{align*}
Since
$$ \lim_{n \to \infty} \left( \frac{1}{n \eta} + \frac{n (1-\eps)^{N+1}}{\eps} \right) = 0, $$
the claim now follows from \eqref{eq:conditionpointwise}.  
\end{proof}

\subsection{Tightness}

This subsection is devoted to proving the following lemma.    

\begin{lemma} \label{lemma:tight}
If 
\begin{equation} \label{eq:sumcongsmall}
	\lim_{\delta \searrow 0} \limsup_{n \to \infty} \Prob \left( \left| \sum_{j=1}^n e^{-i \theta_j^{(n)}} \right| \leq \delta \right) = 0, 
\end{equation}
then, for any $0 < \eps < 1$, the sequence $\frac{1}{n^2} \int_{\mathbb{D}_{1-\eps}} \log^2 |L_n(z)| d \lambda(z)$ is tight.  
\end{lemma} 

The proof of Lemma \ref{lemma:tight} is based on the arguments presented in \cite{K}.

\begin{proof}[Proof of Lemma \ref{lemma:tight}]
Let $0 < \eps < 1$.  Define $R := 1-\eps/2$.  Note that $L_n(z)$ has no poles in the closed disk $\overline{\mathbb{D}}_R$.  Let $\zeta_1^{(n)}, \ldots, \zeta_{k_n}^{(n)}$ denote the zeros of $L_n$ in $\mathbb{D}_R$, where $k_n \leq n$.  So by the Poisson--Jensen formula (see, for instance, \cite[Chapter II.8]{Ma}), for $z:= r e^{i\theta} \in \mathbb{D}_{1 - \eps}$ other than a zero, we have
\begin{equation} \label{eq:poissonjensen}
	\log |L_n(z)| = I_n(z) + \sum_{l=1}^{k_n} \log \left| \frac{ R(z - \zeta_l^{(n)})}{R^2 - \bar{\zeta}_l^{(n)}z} \right|, 
\end{equation}
where
\begin{equation} \label{eq:def:In}
	I_n(z) := \frac{1}{2\pi} \int_{0}^{2 \pi} \log |L_n(Re^{i \phi})| P(z,Re^{i \phi}) d \phi 
\end{equation}
and
\begin{equation} \label{eq:def:P}
	P(z, Re^{i \phi}) := \frac{R^2 - r^2}{R^2 + r^2 - 2 R r \cos(\theta - \phi)}, \qquad r < R. 
\end{equation}

Thus, by the Cauchy--Schwarz inequality, we have
\begin{align*}
	\left( \log |L_n(z)| \right)^2 &\leq 2 | I_n(z) |^2 + 2 \left( \sum_{l=1}^{k_n} \log \left| \frac{ R(z - \zeta_l^{(n)})}{R^2 - \bar{\zeta}_l^{(n)}z} \right| \right)^2 \\
		&\leq 2 | I_n(z) |^2 + 2k_n \sum_{l=1}^{k_n} \log^2 \left| \frac{ R(z - \zeta_l^{(n)})}{R^2 - \bar{\zeta}_l^{(n)}z} \right|.  
\end{align*}
Hence, we conclude that
\begin{align} \label{eq:tightmain}
	\frac{1}{n^2} \int_{\mathbb{D}_{1 - \eps}} \log^2 |L_n(z)| d \lambda(z) &\leq \frac{2}{n^2} \int_{\mathbb{D}_{1-\eps}} |I_n(z)|^2 d \lambda(z) \\
		&\qquad\qquad+ \frac{2k_n}{n^2} \sum_{l=1}^{k_n} \int_{\mathbb{D}_{1 - \eps}} \log^2 \left| \frac{ R(z - \zeta_l^{(n)})}{R^2 - \bar{\zeta}_l^{(n)}z} \right| d \lambda(z). \nonumber
\end{align}

Observe that 
$$ \inf_{z \in \mathbb{D}_{1-\eps}} \inf_{y \in \mathbb{D}_R} |R^2 - yz| \geq C_1, $$
where $C_1 > 0$ depends only on $\eps$.  Similarly, 
$$ \sup_{z \in \mathbb{D}_{1 - \eps}} \sup_{y \in \mathbb{D}_R} |R^2 - yz| \leq 2. $$
Thus, for any $y \in \mathbb{D}_R$, we obtain
\begin{align*}
	\int_{\mathbb{D}_{1- \eps}} \log^2 \left| \frac{R (z-y)}{R^2 - yz} \right| d\lambda(z) &\leq 2 \int_{\mathbb{D}_{1 - \eps}} \log^2 |R(z-y)| d\lambda(z) \\
		&\qquad\qquad + 2 \int_{\mathbb{D}_{1 - \eps}} \log^2 |R^2 - yz| d\lambda(z) \\
		&\leq 4 \pi \log^2 |R| + 4 \int_{\mathbb{D}_{1 - \eps}} \log^2 |z- y| d\lambda(z) \\
		&\qquad\qquad + 4 \int_{\mathbb{D}_{1 - \eps}} \left( \frac{1}{|R^2 - yz|^2} + |R^2 - yz|^2 \right) d \lambda(z) \\
		&\leq 4 \pi \log^2 |R| + 4 \int_{\mathbb{D}_{1 - \eps}} \log^2 |z - y| d \lambda(z) \\
		&\qquad\qquad + 4 \pi \left( \frac{1}{C_1^2} + 4 \right) \\
		&\leq C_2,
\end{align*}
where $C_2 > 0$ depends only on $\eps$.  Here we used that $\log | \cdot |$ is square integrable as well as the bound $\log^2 |x| \leq \frac{2}{|x|^2} + 2 |x|^2$.  Therefore, we conclude that
$$ \sup_{y \in \mathbb{D}_R} \int_{\mathbb{D}_{1 - \eps}} \log^2 \left| \frac{R (z-y)}{R^2 - yz} \right| d\lambda(z) \leq C_2, $$
and hence (since $k_n \leq n$)
$$ \frac{k_n}{n^2} \sum_{l=1}^{k_n} \int_{\mathbb{D}_{1 - \eps}} \log^2 \left| \frac{R(z - \zeta_l^{(n)})}{R^2 - \bar{\zeta}_l^{(n)} z } \right| d \lambda(z) \leq C_2. $$

Thus, in view of \eqref{eq:tightmain}, it suffices to show that
$$ \frac{1}{n^2} \int_{\mathbb{D}_{1 - \eps}} |I_n(z)|^2 d \lambda(z) $$
is tight.  

We recall that, for $z = r e^{i \theta}$, 
$$ I_n(z) = \frac{1}{2 \pi} \int_{0}^{2 \pi} \log |L_n(R e^{i \phi})| \frac{R^2 - r^2}{R^2 + r^2 - 2 R r \cos (\theta - \phi)} d \phi. $$
We now observe that, for any $\theta, \phi \in [0,2 \pi)$ and every $0 \leq r \leq 1 - \eps$, we have
$$ C_3' \leq R^2 - r^2 \leq C_3 $$
and
$$ C_4' \geq R^2 + r^2 - 2 R r \cos(\theta - \phi) \geq (R - r)^2 \geq C_4, $$
where $C_3, C_3', C_4, C_4' > 0$ depend only on $\eps$.  

We write
\begin{align*}
	I_n(z) &= \frac{1}{2 \pi} \int_{0}^{2 \pi} \log_+ |L_n(R e^{i \phi})| P(z,Re^{i\phi}) d \phi \\
		&\qquad\qquad - \frac{1}{2 \pi} \int_{0}^{2 \pi} \log_- |L_n(R e^{i \phi})| P(z,Re^{i\phi}) d \phi. 
\end{align*}
In particular, 
\begin{align*}
	I_n(z) &\leq \frac{1}{2 \pi} \int_{0}^{2 \pi} \log_+ |L_n(R e^{i \phi})| P(z,Re^{i\phi}) d \phi \\
		&\leq \frac{C_3}{C_4} \frac{1}{2 \pi} \int_{0}^{2 \pi} \log_+ |L_n(Re^{i \phi})| d \phi. 
\end{align*}
On the other hand, 
\begin{align*}
	I_n(z) &\geq \frac{C_3'}{C_4'} \frac{1}{2 \pi} \int_{0}^{2 \pi} \log_+ |L_n(R e^{i \phi})| d \phi - \frac{C_3}{C_4} \frac{1}{2 \pi} \int_{0}^{2 \pi} \log_- |L_n(Re^{i \phi})| d \phi \\
		&= \frac{C_3}{C_4} \frac{1}{2 \pi} \int_{0}^{2 \pi} \log |L_n(R e^{i \phi})| d \phi + \left( \frac{C_3'}{C_4'} - \frac{C_3}{C_4} \right) \frac{1}{2 \pi} \int_{0}^{2 \pi} \log_+ |L_n(R e^{i \phi})| d \phi \\
		&\geq \frac{C_3}{C_4} I_n(0) - \left| \frac{C_3'}{C_4'} - \frac{C_3}{C_4} \right| \frac{1}{2 \pi} \int_{0}^{2 \pi} \log_+ |L_n(R e^{i \phi})| d \phi 
\end{align*}
by definition of $I_n(0)$ (see \eqref{eq:def:In} and \eqref{eq:def:P}).  Since 
$$ \log_+ |L_n(Re^{i \phi})| \leq \log_+ \left( \sum_{j=1}^n \left| \frac{1}{ e^{i \theta_j^{(n)}} - R e^{i \phi}} \right| \right) \leq \log \left( \frac{2 n}{\eps} \right) $$
uniformly in $\phi$, we conclude that
$$ I_n(z) \leq \frac{C_3}{C_4} \log \left( \frac{2 n}{\eps} \right) $$
and
$$ I_n(z) \geq \frac{C_3}{C_4} I_n(0) - \left| \frac{C_3'}{C_4'} - \frac{C_3}{C_4} \right| \log \left( \frac{2 n}{\eps} \right) $$
for all $z \in \mathbb{D}_{1 - \eps}$.  

As 
$$ \lim_{n \to \infty} \frac{1}{n} \log \left( \frac{2 n }{\eps} \right) = 0, $$
it suffices to show that $\frac{1}{n} I_n(0)$ is bounded below in probability.  

From \eqref{eq:poissonjensen}, we observe that
$$ I_n(0) \geq \log |L_n(0)| $$
since 
$$ \sum_{l = 1}^{k_n} \log \left| \frac{\zeta_l^{(n)}}{R} \right| \leq 0. $$
Thus, for any $\eta > 0$, we have
\begin{align*}
	\Prob \left( \frac{1}{n} I_n(0) \leq -\eta \right) &\leq \Prob \left( \frac{1}{n} \log |L_n(0)| \leq \eta \right) \\
		&\leq \Prob\left( |L_n(0)| \leq e^{- \eta n} \right) \\
		&= \Prob\left( \left| \sum_{j=1}^n e^{- i \theta_j^{(n)}} \right| \leq e^{- \eta n} \right). 
\end{align*}
From \eqref{eq:sumcongsmall}, we obtain that, for any $\eta > 0$, 
$$ \lim_{n \to \infty} \Prob \left( \frac{1}{n} I_n(0) \leq -\eta \right) = 0. $$

Combining the bounds above, we conclude that the sequence 
$$ \frac{1}{n^2} \int_{ \mathbb{D}_{1 - \eps}} |I_n(z)|^2 d \lambda(z) $$
is tight, and the proof is complete.  
\end{proof}

\subsection{Completing the proof of Lemmas \ref{lemma:main} and \ref{lemma:alternative}}

We now complete the proof of Lemmas \ref{lemma:main} and \ref{lemma:alternative}.  Indeed, in view of Lemma \ref{lemma:TVtight}, the proof reduces to showing that 
\begin{enumerate}[(i)]
\item for a.e. $z \in \mathbb{D}$, $\frac{1}{n} \log |L_n(z)|$ converges in probability to zero as $n \to \infty$,
\item for any $0 < \eps < 1$, the sequence $\frac{1}{n^2} \int_{\mathbb{D}_{1 - \eps}} \log^2 |L_n(z)| d \lambda(z)$ is tight.
\end{enumerate}
Thus, Lemma \ref{lemma:main} follows from Lemmas \ref{lemma:pointwise} and \ref{lemma:tight}.  Lemma \ref{lemma:alternative} follows from Lemma \ref{lemma:tight} as the convergence of $\frac{1}{n} \log |L_n(z)|$ to zero is assumed in the statement of the lemma.

\end{document}